\documentclass[letterpaper, 10 pt, conference]{ieeeconf}  

\IEEEoverridecommandlockouts                              
\usepackage{bbm}
\usepackage{mathrsfs}
\usepackage{verbatim}
\usepackage{amsmath}
\usepackage{amsfonts}
\usepackage{bm}
\usepackage{graphicx}
\usepackage{float}

\usepackage{amsthm}
\usepackage{xcolor}
\usepackage[colorinlistoftodos]{todonotes}
\usepackage{mathtools}
\usepackage{cite}
\usepackage{colortbl}

\theoremstyle{definition}

\newtheorem{definition}{Definition}
\newtheorem{remark}{Remark}
\newtheorem{problem}{Problem}

\theoremstyle{plain}

\newtheorem{theorem}{Theorem}
\newtheorem{lemma}{Lemma}

\title{\LARGE \bf
On Decentralized Minimax Control with Nested Subsystems}

\author{Aditya Dave, {\itshape{Student Member, IEEE,}} Nishanth Venkatesh, {\itshape{Student Member, IEEE,}}  \\
Andreas A. Malikopoulos, {\itshape{Senior Member, IEEE}} 
	\thanks{This research was supported by the Sociotechnical Systems Center (SSC) at the University of Delaware.} %
	\thanks{The authors are with the Department of Mechanical Engineering, University of Delaware, Newark, DE 19716 USA (email: \texttt{adidave@udel.edu; nish@udel.edu; andreas@udel.edu).}} }

\begin{document}

\maketitle
\thispagestyle{empty}

\begin{abstract}
In this paper, we investigate a decentralized control problem with nested subsystems, which is a general model for one-directional communication amongst many subsystems. The noises in our dynamics are modelled as uncertain variables which take values in finite sets. The objective is to minimize a worst-case shared cost. We demonstrate how the prescription approach can simplify the information structure and derive a structural form for optimal control strategies. The structural form allows us to restrict attention to control strategies whose domains do not grow in size with time, and thus, this form can be utilized in systems with long time horizons. Finally, we present a dynamic program to derive the optimal control strategies and validate our results with a numerical example.
\end{abstract}

\section{Introduction}

Modern engineering systems with interdependent subsystems and limited communication have directed attention towards decentralized control problems with cooperative agents.
Such systems include social media platforms \cite{Dave2020SocialMedia}, large hierarchical organizations \cite{zabojnik2002centralized}, and vehicle platoons \cite{mahbub2021_platoonMixed}.
Typically, in decentralized systems, different agents must coordinate their actions with access to different information. Such systems are characterized by their \textit{information structure}, which describes the information available to each agent at each time \cite{mahajan2012}.
The literature in decentralized control has predominantly focused on minimizing the expected cost shared by all agents \cite{nayyar2015structural, Dave2019a, 17, dave2020structural, Dave2021a, Malikopoulos2021, Dave2022a}, or a given norm of a shared cost in linear time-invariant systems \cite{rotkowitz2005characterization, shah2013cal, lessard2012optimal}, for specific information structures. In these problems, the system noises are modeled as random variables with known distributions.

More recently, there has been an interest in decentralized control problems with a worst-case cost \cite{gagrani2017decentralized, gagrani2020worst}. In these problems, we assume no knowledge of the distributions of the noises but consider that their feasible sets are known. Such noises are called \textit{uncertain variables}, which are non-stochastic analogues of random variables \cite{nair2013nonstochastic, rangi2019towards}. The objective is to minimize a maximum possible cost given only the feasible sets for all uncertain variables. This conservative measure of system performance is appropriate for (1) safety-critical applications that require concrete performance guarantees or (2) systems where only feasible sets can be determined a priori for all noises. For centralized worst-case problems, dynamic programs (DPs) can be used to compute the optimal control strategies \cite{bertsekas1973sufficiently, bernhard2003minimax}, but these DPs cannot directly be applied to decentralized problems because of information asymmetry among agents \cite{mahajan2012}. 
This issue has been resolved for systems with partial history sharing among all agents using the common information approach. While this approach is valid for a large number of models \cite{gagrani2017decentralized}, its computational tractability suffers when the private information of agents increases with time, because of the concurrent growth in domains of the optimal control strategies\cite{17}. Increasing private information may occur in systems with one-directional communication \cite{xie2020optimally, Dave2021a} like 
large organizations, where management may not share all observations with employees \cite{zabojnik2002centralized} and vehicle platoons \cite{mahbub2021_platoonMixed}. 
However, with one-directional communication, the information available to a group of agents may be nested within the information available to another group of agents.

In this paper, we analyze a general decentralized system with \textit{nested subsystems}, which allows one-directional communication among agents. We consider multiple subsystems with multiple agents, each of whom partially observes a global state. Agents sequentially share data with all agents within their subsystem. Thus, there is common information within each subsystem and private information available only to individual agents. In our information structure, we consider that the common information in each subsystem is also available to all agents in preceding subsystems, but the common information in preceding subsystems may contain data which is unavailable to agents in a given subsystem. For example, consider a problem with $3$ subsystems. Then, the common information of all agents in subsystem $3$ is a subset of the common information in subsystem $2$, which in turn, is a subset of the common information in subsystem $1$.
Note that partial history sharing is a special case of this model which occurs when the problem has just one subsystem.

Our main contribution in this paper is the derivation of a structural form for optimal control strategies in decentralized minimax problems with nested subsystems (Theorem \ref{struct_result_k}), which can be used to compute the strategies using a DP (Subsection \ref{subsection:DP}). We model all noises using uncertain variables with known feasible sets. In our exposition, we use a combination of \textit{the person-by-person approach} \cite{nayyar2015structural}, and \textit{the prescription approach} \cite{dave2020structural,Dave2021a} to iteratively derive time-invariant domains for optimal strategies. Our analysis is more involved than simply applying the common information approach, which sets up a centralized problem with information available to \textit{all} agents in the system. However, in comparison with the common information approach: (1) our approach requires milder conditions on private information to achieve time invariant domains for optimal control strategies (Remark \ref{remark_common_info}), and (2) thus, our DP improves the computational tractability of a larger class of systems. 
As it is standard in the literature \cite{bertsekas1973sufficiently, gagrani2017decentralized}, we first analyze the terminal cost problem and then extend our results to allow additive costs. 

The remainder of the paper proceeds as follows. In Section \ref{subsection:Formulation}, we formulate the problem. In Section \ref{section:equivalent}, we construct an equivalent problem with a simpler information structure. In Section \ref{section:structural}, we analyze the problem and derive our main results, and in Section \ref{section:example}, we present a numerical example validating our results. Finally, in Section \ref{section:conclusion}, we present concluding remarks and discuss ongoing work.

\section{Problem Formulation} \label{subsection:Formulation}

We consider a decentralized system with a set of subsystems $\mathcal{N} = \{1,\dots,N\}$, $N \in \mathbb{N}$. Each subsystem $n \in \mathcal{N}$ contains $K^n \in \mathbb{N}$ agents in a set $\mathcal{K}^n = \{1,\dots,K^n\}$, and each agent acts across $T \in \mathbb{N}$ discrete time steps.
For each $t=0,\dots, T$, the state of the system is denoted by an uncertain variable $X_t$ which takes values in a finite set $\mathcal{X}_t$ and the control action of each agent $k \in \mathcal{K}^n$, $n \in \mathcal{N},$ is denoted by an uncertain variable $U_t^{k,n}$ which takes values in a finite set $\mathcal{U}_t^{k,n}$. Recall that uncertain variables are non-stochastic analogues of random variables with known feasible sets but unknown distributions \cite{nair2013nonstochastic}.
We denote all actions in a subsystem $n \in \mathcal{N}$ collectively by ${U}_t^{n}:=(U_t^{k,n} : k \in \mathcal{K}^n)$ and all actions in the system by $U_t^{1:N} := (U_t^1, \dots, U_t^N)$.
Starting at $X_0$, the state of the system evolves as 
\begin{equation}
    X_{t+1} = f_t\left(X_t,U_t^{1:N},W_t\right), \quad t =0,\dots,T-1, \label{st_eq}
\end{equation}
where $W_t$ is the disturbance of the system at time $t$ which takes values in a finite set $\mathcal{W}_t$.
At each $t=0,\dots,T$, each agent $k \in \mathcal{K}^n$, $n \in \mathcal{N}$, has an observation $Y_t^{k,n} := h_t^{k,n}(X_t,V_t^{k,n})$, which takes values in a finite set $\mathcal{Y}^{k,n}_t$. Here, $V_t^{k,n}$ is an observation noise which takes values in a finite set $\mathcal{V}^{k,n}_t$. We denote all noises in subsystem $n \in \mathbb{N}$ collectively by $V_t^n := (V_t^{k,n}: k \in \mathcal{K}^n)$ with a feasible set $\mathcal{V}_t^n := \prod_{k \in \mathcal{K}^n} \mathcal{V}_t^{k,n}$.
The disturbances $\{W_t: t=0,\dots,T\}$, measurement noises $\big\{V_t^{k,n}: k \in \mathcal{K}^n, n \in \mathcal{N}, t=0,\dots,T\big\}$, and initial state $X_0$ are collectively called the \textit{primitive variables}. 
Each primitive variable is independent from all other primitive variables. This ensures that the system is Markovian in a non-stochastic sense \cite{nair2013nonstochastic, bertsekas1973sufficiently}. Next, we describe the information structure of the system.

\begin{definition}
The \textit{memory} of agent $k \in \mathcal{K}^n$, $n \in \mathcal{N}$, at each $t$, is a set $M_t^{k,n} \subseteq \{Y_{0:t-1}^{i,m}, U_{0:t-1}^{i,m}: i \in \mathcal{K}^m, m \geq n\}$, which takes values in a finite collection of sets $\mathcal{M}_t^{k,n}$ and satisfies \textit{perfect recall}, i.e., $M_t^{k,n} \subseteq M_{t+1}^{k,n}$ for all $t=0,\dots,T-1$.
\end{definition}

\begin{remark}
To be consistent with the exposition in the literature \cite{gagrani2017decentralized, nayyar2015structural}, we consider that for all $k \in \mathcal{K}^n$, $n \in \mathcal{N}$, and $t=0,\dots,T$, the memory $M_{t}^{k,n}$ is updated before the observation $Y_t^{k,n}$ is realized. 
\end{remark}

In an information structure with nested subsystems, we partition the memory $M_t^{k,n}$ of each agent $k \in \mathcal{K}^n$, $n \in \mathcal{N}$, into two components, $C_t^n$ and $L_t^{k,n}$, with the following properties:

\textit{1) The common information} of all agents in a subsystem $n \in \mathcal{N}$ at each $t = 0,\dots,T$ is the information available to all agents in $\mathcal{K}^n$. We define it as the set $C_t^n := \bigcap_{k \in \mathcal{K}^n} M_t^{k,n}$ which takes values in a finite collection of sets $\mathcal{C}_t^n$, and satisfies the following properties: (1) \textit{nestedness}, i.e., $C_t^m \subseteq C_t^n$ for all $m \in \mathcal{N}$ with $m > n$ and $t = 0,\dots,T$, and (2) \textit{perfect recall}, i.e., $C_{t}^n \subseteq C_{t+1}^n$ for all $t$. The property of nestedness implies that the common information of a subsystem is also available to all preceding subsystems.

\textit{2) The private information} of an agent $k \in \mathcal{K}^n$, $n \in \mathcal{N}$, at each $t = 0,\dots,T$ is the information available only to agent $k$. We define it as the set $L_t^{k,n} := M_t^{k,n} \setminus C_t^n$ which takes values in a finite collection of sets $\mathcal{L}_t^{k,n}$. For any pair of subsystems $n,m \in \mathcal{N}$ with $m > n$, we impose the condition that $L_t^{k,m} \cap C_t^n = \emptyset$ for all $k \in \mathcal{K}^m$, i.e., no element in the private information of an agent can be available to \textit{all} agents of a preceding subsystem. 



For each $n \in \mathcal{N}$, we also define the \textit{new information} added to $C_t^n$ at each $t=0,\dots,T$ as a set $Z_t^n := C_{t}^n \setminus C_{t-1}^n$, which takes values in a finite collection of sets $\mathcal{Z}_t^n$, where $C_{-1}^n := \emptyset$. The new information is also nested, i.e., $Z_t^m \subseteq Z_t^n$, for all $m, n \in \mathcal{N}$, $m > n$, and $t$. Thus, the new information available to any subsystem is a subset of the new information available to any preceding subsystem at each $t$. As an illustration, consider a system with three agents divided into two subsystems in Fig. \ref{fig:illustration}. Subsystem $1$ contains two agents and subsystem $2$ contains one agent. The red dashed arrows indicate the communication between the agents at time $t$. The one-directional, dashed arrow indicates sharing of common information from subsystem $2$ to all agents in subsystem $1$ at time $t$, which ensures that $C_t^2 \subseteq C_t^1$.

\begin{figure}[ht!]
  \centering
  \includegraphics[width=0.8\linewidth, keepaspectratio]{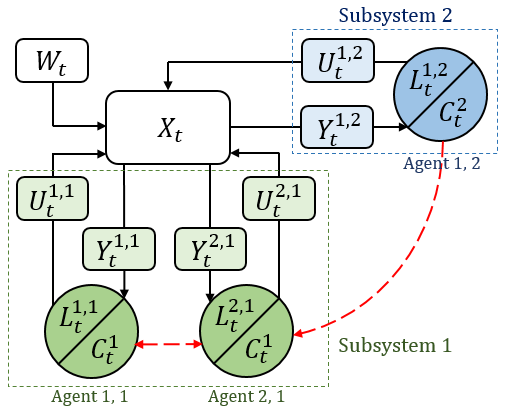}
  \caption{A system with three agents in two nested subsystems at time $t$}
  \label{fig:illustration}
\end{figure}

After updating their memory and realizing their observations at each $t=0,\dots,T$, each agent $k \in \mathcal{K}^n$, $n \in \mathcal{N}$, selects an action
\begin{gather} \label{U_def}
    U_t^{k,n} = g_t^{k,n}(Y_t^{k,n},M_t^{k,n}) = g_t^{k,n}(Y_t^{k,n},L_t^{k,n}, C_t^n),
\end{gather}
where $M_t^{k,n} = \{L_t^{k,n}, C_t^n\}$, and $g_t^{k,n}: \mathcal{Y}_t^{k,n} \times \mathcal{M}_t^{k,n} \to \mathcal{U}_t^{k,n}$ is the \textit{control law} of agent $k$ in subsystem $n$ at time $t$. The \textit{control strategy} of subsystem $n \in \mathcal{N}$ is $\boldsymbol{g}^{n} := (g_{0:T}^{k,n}: k \in \mathcal{K}^n)$ and the control strategy of the system is $\boldsymbol{g} := \big(\boldsymbol{g}^{1}, \dots,\boldsymbol{g}^{N}\big)$. The set of all feasible strategy profiles is $\mathcal{G}$. All agents collectively incur a shared terminal cost $d_T(X_T) \in \mathbb{R}_{\geq0}$ at the time horizon $T$.
The system performance is measured by the worst-case terminal cost
\begin{gather}
    \mathcal{J}(\boldsymbol{g}) := \max_{X_0, W_{0:T}, V_{0:T}^{1:N}} d_T(X_T).
\end{gather}


\begin{problem} \label{problem_1}
The optimization problem is $\min_{\boldsymbol{g} \in \mathcal{G}} \mathcal{J}(\boldsymbol{g}),$
given the feasible sets $\{\mathcal{X}_0, \mathcal{W}_{t}, \mathcal{V}_{t}^{n} : n \in \mathcal{N}, t = 0,\dots,T\}$, the cost function $d_T$, and the system dynamics $\{f_t, h_t^{k,n}: k \in \mathcal{K}^n, n \in \mathcal{N}, t=0,\dots,T\}$.
\end{problem}

Our aim is to develop a DP which can derive an optimal strategy profile $\boldsymbol{g}^* \in \mathcal{G}$ for Problem \ref{problem_1}. Note that an optimal strategy is guaranteed to exist because all variables take values in finite sets.

\begin{remark}
Variations of many existing information structures are special cases of our model. For example, a system with partial history sharing \cite{gagrani2017decentralized, 17}, emerges from our model when all agents belong to a single subsystem, i.e., we set $N = 1$. Similarly, a system with one-directional delayed communication among a chain of $K \in \mathbb{N}$ agents is the special case where we have one agent per subsystem, i.e., $N=K$ and $|\mathcal{K}^n| = 1$ for all $n \in \mathcal{N}$. Simpler cases of instantaneous one-directional communication among just two agents were considered in \cite{Dave2021a,xie2020optimally} for stochastic control.
\end{remark}






\section{Equivalent System with Nested Information} \label{section:equivalent}

In this section, we construct an equivalent system with a simpler information structure, where each subsystem acts as the decision maker. For example, consider two distinct agents $k,i \in \mathcal{K}^n$ in a subsystem $n \in \mathcal{N}$. In general, given a strategy $\boldsymbol{g}$, agent $k$ cannot generate the action $U_t^{i,n}$ of agent $i$ since they cannot access the private information $L_t^{i,n}$. However, all agents in subsystem $n \in \mathcal{N}$ can access the common information $C_t^n$. This motivates us to consider a two-stage mechanism for action generation: (1) each subsystem $n$ generates a \textit{partial action} for each agent in $\mathcal{K}^n$ using only the common information $C_t^n$, and then (2) each agent $k \in \mathcal{K}^n$ uses the partial action along with their private information $L_t^{k,n}$ to generate their action $U_t^k$.


\begin{definition}
A \textit{partial action} for agent $k \in \mathcal{K}^n$, $n \in \mathcal{N},$ at each $t =0,\dots,T$ is a mapping  $\hat{U}_t^{k,n}:\mathcal{Y}_t^{k,n} \times \mathcal{L}_t^{k,n} \to \mathcal{U}_t^{k,n}$ where $\hat{\mathcal{U}}_t^{k,n}$ is a finite set.
\end{definition}

At each $t=0,\dots,T$, a subsystem $n \in \mathcal{N}$ generates a partial action for an agent $k \in \mathcal{K}^n$ using a partial control law $\hat{g}_t^{k,n}: \mathcal{C}_t^n \to \hat{\mathcal{U}}_t^{k,n}$, which yields $\hat{U}_t^{k,n} = \hat{g}_t^{k,n}(C_t^n)$. We define the partial control law for subsystem $n$ at time $t$ as $\hat{g}_t^n := (g_t^{k,n} : k \in \mathcal{K}^n)$ and the partial control strategy for the system as $\hat{\boldsymbol{g}} := (\hat{g}_{0:T}^{1}, \dots, \hat{g}_{0:T}^N)$. To simplify the notation, we also define $\hat{U}_t^{n} := \big(\hat{U}_t^{k,n} : k \in \mathcal{K}^n\big)$ and compactly write $\hat{U}_t^n = \hat{g}_t^n(C_t^n)$ for each $t = 0,\dots,T$. The partial action $\hat{U}_t^n$ is generated using only the common information in subsystem $n$ and thus is available to all agents in $\mathcal{K}^n$. Furthermore, given a partial action $\hat{U}_t^n$, each agent $k \in \mathcal{K}^n$ must generate their action as $U_t^{k,n} = \hat{U}_t^{k,n}({Y}_t^{k,n},L_t^{k,n})$. Next, we construct a new state which is controlled only using the partial actions of all subsystems. For each $n \in \mathcal{N}$, let ${L}_t^n := \{Y_t^{k,n}, L_t^{k,n}: k \in \mathcal{K}^n\}$. Then, the new state is
\begin{gather}
    \hat{X}_t := \big\{X_t, {L}_t^1, \dots, {L}_t^N\big\}, \quad t = 0,\dots,T, \label{hat_X_def}
\end{gather}
and it takes values in a finite collection of sets $\hat{\mathcal{X}}_t$.


\begin{lemma} \label{eq_St}
For each $t = 0, \dots, T-1$, we can construct a state evolution function $\hat{f}_t(\cdot)$ and an observation rule $\hat{h}^{{n}}_t(\cdot)$, $n \in \mathcal{N}$, such that $\hat{X}_{t+1} = \hat{f}_{t}(\hat{X}_t, \hat{U}_t^{1:N}, W_{t}, V_{t+1}^{1:N})$ and $Z_{t+1}^n = \hat{h}_{t}^n( \hat{X}_t, \hat{U}_{t}^{1:N})$, respectively. In addition, we can construct a terminal cost function $\hat{d}_T(\cdot)$ which yields a shared cost $\hat{d}_T(\hat{X}_T) = d_T(X_T).$
\end{lemma}

\begin{proof}
We first show how $\hat{f}_t(\cdot)$ can be constructed at each $t = 0,\dots,T$ by establishing that each component in $\hat{X}_{t+1} = \big\{X_{t+1}, L_{t+1}^1,\dots, L_{t+1}^N \big\}$ can be written in terms of the variables in the RHS. Note that $X_{t+1} = f_t(X_t, U_t^{1:N}, W_t)$, where $U_t^{k,n} = \hat{U}_t^{k,n}({Y}_t^{k,n},L_t^{k,n})$ for all $k \in \mathcal{K}^n$, $n\in \mathcal{N}$. For each $n \in \mathcal{N}$, $L_{t+1}^{n} \subseteq \{L_t^m, Y_{t+1}^{k,m}, {U}_t^{k,m} : k \in \mathcal{K}^m,$ $m \geq n\}$ and for each $k \in \mathcal{K}^n$, $Y_{t+1}^{k,n} = h_{t+1}^{k,n}(X_{t+1},V_{t+1}^{k,n})$. Using these equations, each term in $\hat{X}_{t+1}$ can be written using the variables in the RHS. The proof is complete by defining an appropriate function $\hat{f}_t(\cdot)$. The other functions can also be constructed in a similar manner.
\end{proof}

Lemma \ref{eq_St} yields a ``new" decentralized system with a shared state $\hat{X}_t$, action $\hat{U}_t^n$, and observation $Z_t^n$ for each subsystem $n \in \mathcal{N}$ and $t=0,\dots,T$. The cost is $\hat{d}_T(\hat{X}_T)$ and the performance criterion is $\hat{\mathcal{J}}(\hat{\boldsymbol{g}}) := \max_{X_0, W_{0:T}, V_{0:T}^{1:N}} \hat{d}_T(\hat{X}_T).$ This new decentralized system has a nested information structure \cite{Dave2021a} since at each $t$, a partial action $\hat{U}_t^n$ is generated using only $C_t^n = Z_{0:t}^n$, while $Z_t^m \subseteq Z_t^n$ for all $n,m \in \mathcal{N}$ with $m > n$. Next, we define a new optimization problem and show that it is equivalent to Problem \ref{problem_1}.

\begin{problem} \label{problem_2}
The new problem is $\min_{\hat{\boldsymbol{g}} \in \hat{\mathcal{G}}} \hat{\mathcal{J}}(\hat{\boldsymbol{g}}),$
given the sets $\{{\mathcal{X}}_0, {\mathcal{W}}_{t}, {\mathcal{V}}_{t}^{n} : n \in \mathcal{N}, t = 0,\dots,T\}$, the cost function $\hat{d}_T$, and the dynamics $\{\hat{f}_t, \hat{h}_t^{n} : k \in \mathcal{K}^n, n \in \mathcal{N}, t=0,\dots,T\}$.
\end{problem}


\begin{lemma} \label{lem_hatg_g_relation} 
For any given control strategy $\boldsymbol{g}$, consider a partial control strategy $\hat{\boldsymbol{g}}$
\begin{gather}
    \hat{g}_t^{k,n}(C_t^n)(\cdot,\cdot) := g_t^{k,n}(\cdot, \cdot, C_t^{n}), \; t = 0,\dots,T, \label{u_presc}
\end{gather}
for all $k \in \mathcal{K}^n$, $n \in \mathcal{N}$. Then, $\mathcal{J}(\boldsymbol{g}) = \hat{\mathcal{J}}(\hat{\boldsymbol{g}})$.
Moreover, for any given partial control strategy $\hat{\boldsymbol{g}}$, we can construct a control strategy $\boldsymbol{g}$ such that
\begin{gather}
     g_t^{k,n}(\cdot,\cdot,C_t^{n}) := \hat{g}_t^{k,n}(C_t^n)(\cdot,\cdot), \; t = 0,\dots,T, \label{u_presc_inv}
\end{gather}
for all $k \in \mathcal{K}^n$, $n \in \mathcal{N}$. Then, $\hat{\mathcal{J}}(\hat{\boldsymbol{g}}) = \mathcal{J}(\boldsymbol{g})$.
\end{lemma}

\begin{proof}
For the first part, given a control strategy $\boldsymbol{g}$ and the partial control strategy $\hat{\boldsymbol{g}}$, the action at each $t$ for any agent $k \in \mathcal{K}^n$, $n \in \mathcal{N}$, is $U_t^{k,n} = g_t^{k,n}({Y}_t^{k,n},L_t^{k,n}, C_t^n) = \hat{g}_t^{k,n}(C_t^n)({Y}_t^{k,n},L_t^{k,n})$. Thus, the control law and partial control law result in the same action $U_t^{k,n}$ for any given $\{{Y}_t^{k,n},L_t^{k,n}, C_t^n\}$. Subsequently, for any given realizations of all primitive variables $\{X_0, W_{0:T}, V_{0:T}^{1:3}\}$, the two strategies yield the same terminal state $X_T$ and cost $d_T(X_T)$.
The proof of the second part follows using arguments similar to the first part.
\end{proof}

Lemma \ref{lem_hatg_g_relation} implies that given an optimal partial strategy $\hat{\boldsymbol{g}}^*$ for Problem \ref{problem_2}, a control strategy $\boldsymbol{g}^*$ constructed by \eqref{u_presc_inv} is an optimal solution to Problem \ref{problem_1}, with performance $\mathcal{J}(\boldsymbol{g}^*)$. Similarly, an optimal strategy for Problem \ref{problem_1} yields an optimal solution to Problem \ref{problem_2}. Thus, the two problems are equivalent.


\begin{remark}
It is easier to analyze the equivalent system with a nested information structure than the original system. In the next section, we use this new information structure to derive results for optimal partial strategies, and then use Lemma \ref{lem_hatg_g_relation} to characterize the optimal control strategies.
\end{remark}

\section{Main Results} \label{section:structural}

In this section, we derive a structural form for optimal partial strategies in Problem \ref{problem_2}. To simplify the exposition, we carry out the detailed derivation for a system with only $3$ subsystems. This illustrates all the key arguments required to prove the results for different cases. Later, in subsection IV-D, we present the results for $N$ subsystems.

\subsection{Analysis for Subsystem 1} \label{pbp_agent_1}

In this subsection, we analyze the optimal partial strategy of subsystem $1$ in a system with $3$ subsystems, i.e., $\mathcal{N} = \{1,2,3\}$. We use the person-by-person approach and thus, arbitrarily fix the partial control strategies $\hat{\boldsymbol{g}}^2$ and $\hat{\boldsymbol{g}}^3$, such that
$\hat{U}_t^n = \hat{g}_t^n(C_t^n)$, for all $n =2,3$ and $t = 0,\dots,T$, and set up a centralized problem for subsystem $1$. Recall that $C_t^n \subseteq C_t^1$, and thus, any agent in subsystem $1$ can derive the partial action $\hat{U}_t^n$ for all $n = 2,3$ and $t=0,\dots,T$. Next, we construct a centralized problem for subsystem $1$ with a new state $S_t^1 := \big\{\hat{X}_t, C_t^2\big\}$ for each $t=0,\dots,T$, which takes values in a finite collection of sets $\mathcal{S}_t^1$.

\begin{lemma} \label{lem_pbp_suff}
For any given partial strategies $\hat{\boldsymbol{g}}^2$ and $\hat{\boldsymbol{g}}^3$, we can construct a state evolution function $\bar{f}^{{1}}_t(\cdot)$ and an observation rule $\bar{h}^{{1}}_t(\cdot)$, such that ${S}^{{1}}_{t+1} = \bar{f}^{{1}}_{t}({S}^{{1}}_t, \hat{U}_t^{{1}}, W_{t}, V_{t+1}^{1:3})$, and $Z^1_{t+1} = \bar{h}^1_t({S}^{{1}}_t, \hat{U}_t^{{1}})$, respectively, for all $t = 0,\dots,T-1$. In addition, we can construct a terminal cost function $\bar{d}^{{1}}_T(\cdot)$, which yields a shared cost
	$\bar{d}^{{1}}_T({S}^{{1}}_T) = \hat{d}_T(\hat{X}_T).$
\end{lemma}

\begin{proof}
We first show how $\bar{f}_t^1(\cdot)$ can be constructed at each $t=0,\dots,T$ by establishing that each component in $S_{t+1}^1 = \big\{\hat{X}_{t+1}, C_{t+1}^2 \big\}$ can be written in terms of the variables in the RHS. Note that $\hat{X}_{t+1} = \hat{f}_t(\hat{X}_t,\hat{U}_t^1, \hat{g}_t^2(C_t^2), \hat{g}_t^3(C_t^3),$ $W_t, V_{t+1}^{1:3})$, where $\hat{g}_t^2(\cdot)$ and $\hat{g}_t^3(\cdot)$ are given and $C_t^3 \subseteq C_t^2$. Furthermore, $C_{t+1}^2 =C_t^2 \cup Z_{t+1}^2$, where $Z_{t+1}^2 = \hat{h}_t^2(\hat{X}_t, \hat{U}_t^1, \hat{g}_t^2(C_t^2), \hat{g}_t^3(C_t^3))$. Thus, given $\hat{\boldsymbol{g}}^{2}$ and $\hat{\boldsymbol{g}}^{3}$, each term in $S_{t+1}^1$ can be written as a function of the variables in the argument of the RHS. The proof is complete by defining an appropriate function $\bar{f}_t^1(\cdot)$. The other functions can be constructed similarly by rewriting the LHS as a function of the variables in the argument in the RHS.
\end{proof}

Given the partial control strategies $\hat{\boldsymbol{g}}^{2}$ and $\hat{\boldsymbol{g}}^{3}$, Lemma \ref{lem_pbp_suff} yields a new \textit{centralized} system with state $S_t^1$, action $\hat{U}_t^1$, and observation $Z_{t}^1$ for all $t=0,\dots,T$. The terminal cost incurred by the system is $\bar{d}_T^1(S_T^1)$ and the worst-case performance criterion is $\mathcal{J}^1(\hat{\boldsymbol{g}}^{1}) := \max_{X_0, W_{0:T}, V_{0:T}^{1:3}} \bar{d}_T^1(S_T^1).$


\begin{problem} \label{problem_3}
The problem for subsystem $1$ is
    $\min_{\hat{\boldsymbol{g}}^{1}} \mathcal{J}^1(\hat{\boldsymbol{g}}^{1})$,
given the control partial strategies $\hat{\boldsymbol{g}}^{2}$ and $\hat{\boldsymbol{g}}^{3}$, the feasible sets $\{\mathcal{X}_0,\mathcal{W}_{t},\mathcal{V}_{t}^{n} : n = 1,2,3, \; t = 0,\dots,T\}$, the cost function $\bar{d}^1_T$, and the dynamics $\{\bar{f}^1_t,\bar{h}_t^{1} :t=0,\dots,T\}$.
\end{problem}

\begin{remark} \label{remark_5}
Using Lemma \ref{lem_pbp_suff}, we construct a cost $\bar{d}_T^1(S_T^1)$ for the centralized problem such that $\bar{d}_T^1(S_T^1) = \hat{d}_T(\hat{X}_T)$ for all realizations of the primitive variables $\{X_0, W_{0:T}, V_{0:T}^{1:3}\}$. Thus, the performance $\mathcal{J}^1(\hat{\boldsymbol{g}}^1)$ for any given $\hat{\boldsymbol{g}}^2$ and $\hat{\boldsymbol{g}}^3$ is equal to $\mathcal{J}(\hat{\boldsymbol{g}}^{1}, \hat{\boldsymbol{g}}^{2},\hat{\boldsymbol{g}}^{3})$.
\end{remark}

Note that Problem \ref{problem_3} is a partially observed centralized problem with subsystem $1$ as the sole decision maker. Specifically, at each $t=0,\dots,T$, the component $C_t^2$ of the state ${S}_t^1$ is completely observed by subsystem $1$, whereas, the component $\hat{X}_t$ must be estimated using the common information $C_t^1$ and history of partial actions $\hat{U}_{0:t-1}^{1:3}$. For partially observed centralized problems, e.g., Problem \ref{problem_3}, a DP was provided in \cite{bertsekas1973sufficiently} and its complexity was reduced using a \textit{sufficiently informative function}. We simply call this an \textit{information state} to be consistent with its stochastic counterpart. In minimax problems, an information state at time $t$ is a set of feasible values for the partially observed state which are compatible with the information available to a subsystem at time $t$. Since subsystem $1$ completely observes $C_t^2$, we need to define an information state only for the component $\hat{X}_t$. 
For all $t=0,\dots,T$, given the realizations $c_t^1$ and $\hat{u}_{0:t-1}^{1:3}$ of $C_t^1$ and $\hat{U}_{0:t-1}^{1:3}$, respectively,
the realized information state of subsystem $1$ is the set
\begin{multline} \label{pi_1_def}
    P_t^1 
    := \big\{\hat{x}_t \in \hat{\mathcal{X}}_t ~|~ \exists \; \big( x_0 \in \mathcal{X}_0, w_{0:t-1} \in \prod_{\ell=0}^{t-1} \mathcal{W}_{\ell},
    v_{0:t}^{n} \in \\ \prod_{\ell=0}^{t} \mathcal{V}_{\ell}^{n}, \; n = 1,2,3\big)  
    \text{ such that } s_{\ell+1}^1 = \bar{f}_\ell^1\big(s_{\ell}^1,\hat{u}_\ell^1,  w_\ell, v_{\ell+1}^{1:3}\big),\\ 
    z_{\ell+1}^1 = \bar{h}_\ell^1\big(s_{\ell}^1, \hat{u}_\ell^1\big), \; \ell = 0, \dots, t-1
    \big\}.
\end{multline}
The information state each $t$ is a function of uncertain variables and thus, in general, we denote it with the set-valued uncertain variable $\Pi_t^1$ which takes values in a finite collection of feasible sets $\mathcal{P}^1_t \subseteq 2^{\hat{\mathcal{X}}_t}$, where $2^{\hat{\mathcal{X}}_t}$ denotes the power set of $\hat{\mathcal{X}}_t$. Next, we show that the information state $\Pi_t^1$ does not depend on the choice of partial strategy $\hat{\boldsymbol{g}}$.

\begin{lemma} \label{pi_1_evol}
At each $t = 0,\dots,T-1$, there exists a function $\tilde{f}_t^{{1}}(\cdot)$ independent of $\hat{\boldsymbol{g}}$, such that
\begin{equation} \label{pi_1_relation}
    \Pi_{t+1}^{{1}} = \tilde{f}_{t}^{{1}}(\Pi_t^{{1}},\hat{U}_t^{1:3},Z_{t+1}^{{1}}).
\end{equation}
\end{lemma}

\begin{proof}
The proof follows the same arguments as the ones of Lemma \ref{pi_2_evol} in Section IV-B.
\end{proof}

Lemma \ref{pi_1_evol} establishes the Markovian and strategy independent evolution of $\Pi_t^1$ at each $t$. Thus, we state the following result for Problem \ref{problem_3} using the centralized DP in \cite{bertsekas1973sufficiently}. 

\begin{theorem} \label{pbp_struct_result}
    For any given partial strategies $\hat{\boldsymbol{g}}^{2}$ and $\hat{\boldsymbol{g}}^{3}$ of subsystems $2$ and $3$, respectively, in Problem \ref{problem_3}, without loss of optimality, we can restrict attention to partial strategies $\hat{\boldsymbol{g}}^{*1}$ with the structural form
    \begin{gather} \label{eq_pbp_struct_result}
        \hat{U}_t^{1} = \hat{g}_t^{*1}\big(\Pi_t^1, C_t^2\big), \quad t = 0,\dots,T.
    \end{gather}
\end{theorem}

\begin{proof}
This result follows from standard arguments for centralized minimax control problems in \cite{bertsekas1973sufficiently}.
\end{proof}

Theorem \ref{pbp_struct_result} establishes a structural form for an optimal partial strategy $\hat{\boldsymbol{g}}^{*1}$ of subsystem $1$ in Problem $3$, which holds for all $\hat{\boldsymbol{g}}^2$ and $\hat{\boldsymbol{g}}^3$. From Remark \ref{remark_5}, we note that any globally optimal partial strategy $(\hat{\boldsymbol{g}}^{*1}, \hat{\boldsymbol{g}}^{*2}, \hat{\boldsymbol{g}}^{*3})$ for Problem \ref{problem_2}, must necessarily be a person-by-person optimal solution of Problem \ref{problem_3}, i.e., after fixing $\hat{\boldsymbol{g}}^{*2}$ and $\hat{\boldsymbol{g}}^{*3}$ for subsystems $2$ and $3$, respectively, the strategy $\hat{\boldsymbol{g}}^{*1}$ is optimal for Problem \ref{problem_3}. 
Thus, there exists a globally optimal partial strategy $(\hat{\boldsymbol{g}}^{*1}, \hat{\boldsymbol{g}}^{*2}, \hat{\boldsymbol{g}}^{*3})$ where $\hat{\boldsymbol{g}}^{*1}$ takes the structural form in \eqref{eq_pbp_struct_result}. 

\subsection{Analysis for Subsystem 2}

In this subsection, we restrict attention to partial control strategies $\hat{\boldsymbol{g}}^1$ of agent $1$ which satisfy \eqref{eq_pbp_struct_result}, and derive a structural form for the partial strategy of subsystem $2$. Note that given $\hat{\boldsymbol{g}}^1$, agents in subsystem $2$ cannot generate $\hat{U}_t^1$ at any $t$ because they cannot access $\Pi_t^1$. Thus, before using the person-by-person approach for subsystem $2$, we consider that $\hat{U}_t^1$ is generated in two stages for all $t=0,\dots,T$: (1) subsystem $2$ generates a \textit{prescription} for subsystem $1$ using only $C_t^2$, and then (2) subsystem $1$ uses the prescription and its information state $\Pi_t^1$ to generate its partial action $\hat{U}_t^1$.

\begin{definition}
A \textit{prescription} by subsystem $2$ for subsystem $1$ at any $t=0,\dots,T$ is a mapping $\Gamma_t^{[2,1]}:\mathcal{P}^{1}_t \to \mathcal{U}_t^1$ which takes values in a finite set $\mathcal{F}_t^{[2,1]}$.
\end{definition}

At each $t=0,\dots,T$, subsystem $2$ generates a prescription for subsystem $1$ using a prescription law $\psi_t^{[2,1]}: \mathcal{C}_t^2 \to \mathcal{F}_t^{[2,1]}$, which yields $\Gamma_t^{[2,1]} = \psi_t^{[2,1]}(C_t^2)$. We define the prescription strategy for subsystem $2$ as $\boldsymbol{\psi}^2 := (\psi_t^{[2,1]} : t=0,\dots,T)$. Since the prescription $\Gamma_t^{[2,1]}$ is generated only using $C_t^2 \subseteq C_t^1$, it is available to both subsystems $1$ and $2$. Then, subsystem $1$ must generate its partial action as $\hat{U}_t^1 = \Gamma_t^{[2,1]}\big(\Pi_t^1\big)$. Note that in this formulation, the prescription $\Gamma_t^{[2,1]}$ is generated by subsystem $2$ and subsystem $1$ simply utilizes the given prescription to obtain $\hat{U}_t^1$. Thus, 
we can consider that at each $t=0,\dots,T$, subsystem $2$ selects a \textit{complete action} $\Theta_t^2 := (\Gamma_t^{[2,1]}, U_t^2)$. This motivates us to simultaneously derive a structural form for the optimal prescription and partial strategies of subsystem $2$. To this end, we use the person-by-person approach by fixing the partial control strategy $\hat{\boldsymbol{g}}^3$ of subsystem $3$ and noting that subsystem $2$ can derive $\hat{U}_t^3$ for all $t$ because $C_t^3 \subseteq C_t^2$.
Next, we construct a centralized problem for subsystem $2$ with a state 
$S_t^2 := \{\hat{X}_t, \Pi_t^1, C_t^{3}\}$ at each $t$,
which takes values in a finite collection of sets ${\mathcal{S}}_t^2$. Using the same arguments as Lemma \ref{lem_pbp_suff}, we can construct a state evolution function $\bar{f}^2_t(\cdot)$ such that ${S}^{{2}}_{t+1} = \bar{f}^2_{t}({S}^{{2}}_t, \Theta_t^2, W_{t}, V_{t+1}^{1:3})$, and an observation rule $\bar{h}^{{2}}_t(\cdot)$ which yields $Z^2_{t+1} = \bar{h}^2_t({S}^{{2}}_t, \Theta_t^2)$ to obtain a centralized problem with state $S_t^2$, observation $Z_t^2$ and complete action $\Theta_t^2 = (\Gamma_t^{[2,1]}, U_t^2)$ for all $t$. Furthermore, we can construct a cost function $\bar{d}^{{2}}_T(\cdot)$ which yields the terminal cost $\bar{d}^{{2}}_T({S}^{{2}}_T) = \hat{d}_T(\hat{X}_T)$ and a corresponding performance criterion $\mathcal{J}^2(\boldsymbol{\psi}^2, \hat{\boldsymbol{g}}^2) := \max_{X_0, W_{0:T}, V_{0:T}^{1:3}} \bar{d}_T^2(S_T^2).$

\begin{problem} \label{problem_4}
The problem for subsystem $2$ is $\inf_{\boldsymbol{\psi}^2, \hat{\boldsymbol{g}}^2} \mathcal{J}^2(\boldsymbol{\psi}^2, \hat{\boldsymbol{g}}^2)$, given a partial control strategy $\hat{\boldsymbol{g}}^{3}$, the feasible sets $\{\mathcal{X}_0,\mathcal{W}_{t},\mathcal{V}_{t}^{n} : n = 1,2,3, t = 0,\dots,T\}$, the cost function $\bar{d}^2_T$, and the dynamics $\{\bar{f}^2_t,\bar{h}_t^{2}:t=0,\dots,T\}$.
\end{problem}

\begin{remark} \label{remark_4}
Consider a fixed partial control strategy $\hat{\boldsymbol{g}}^3$ for agent $3$. Using the same sequence of arguments as Lemma \ref{lem_hatg_g_relation}, for each prescription strategy $\boldsymbol{\psi}^2$, we can construct a partial control strategy $\hat{\boldsymbol{g}}^1$ such that the performance criterion $\mathcal{J}(\hat{\boldsymbol{g}}^1, \hat{\boldsymbol{g}}^2, \hat{\boldsymbol{g}}^3)$ for Problem \ref{problem_2} is equal to $\mathcal{J}^2(\boldsymbol{\psi}^2, \hat{\boldsymbol{g}}^2)$. Similarly, for each partial control strategy $\hat{\boldsymbol{g}}^1$ we can construct a prescription strategy $\boldsymbol{\psi}^2$ which yields the same performance. Thus, it is equivalent to select either a partial control strategy $\hat{\boldsymbol{g}}^1$ or a prescription strategy $\boldsymbol{\psi}^2$.
\end{remark}


Problem \ref{problem_4} is a partially observed centralized problem with subsystem $2$ as the sole decision maker. Specifically, at each $t=0,\dots,T$, the component $C_t^3$ in state $S_t^2$ is perfectly observed by subsystem $2$, whereas $\hat{X}_t$ and $\Pi_t^1$ must be estimated using an information state like \eqref{pi_1_def}. 
Given the realizations $c_t^2$, $\gamma_{0:t-1}^{[2,1]}$, and $\hat{u}_{0:t-1}^{2:3}$ of $C_t^2$, $\Gamma_{0:t-1}^{[2,1]}$, and $\hat{U}_{0:t-1}^{2:3}$, respectively, the realized information state of subsystem $2$ is
\begin{multline} \label{pi_2_def}
    P_t^2 
    := \big\{\hat{x}_t \in \hat{\mathcal{X}}_t, P_t^1 \in \mathcal{P}_t^1 ~|~ \exists \; \big(x_0 \in \mathcal{X}_0,
    w_{0:t-1} \in \\ \prod_{\ell=0}^{t-1} \mathcal{W}_{\ell}, v_{0:t}^{n} \in \prod_{\ell=0}^{t} \mathcal{V}_{\ell}^{n}, \; n = 1,2,3\big) \text{ s.t. } s_{\ell+1}^2 = \bar{f}_\ell^2\big(s_{\ell}^2,\\ \theta_\ell^2, w_\ell, v_{\ell+1}^{1:3}\big),
    z_{\ell+1}^2 = \bar{h}_\ell^2\big(s_{\ell}^2,\theta_\ell^2\big),  \ell = 0, \dots, t-1
    \big\},
\end{multline}
Thus, the information state of subsystem $2$ is a set-valued uncertain variable $\Pi_t^2$ which takes values in a finite collection of feasible sets $\mathcal{P}^2_t \subseteq 2^{\hat{\mathcal{X}}_t} \times 2^{\mathcal{P}^1_t}$. Next, we show that the information state $\Pi_t^2$ is independent of the choice of $\hat{\boldsymbol{g}}$.

\begin{lemma} \label{pi_2_evol}
At each $t = 0,\dots,T-1$, there exists a function $\tilde{f}_t^2(\cdot)$ independent from $\hat{\boldsymbol{g}}$, such that
\begin{equation} \label{pi_2_relation}
    \Pi_{t+1}^2 = \tilde{f}_{t}^2(\Pi_t^2,\Gamma_t^{[2,1]},\hat{U}_t^{2:3},Z_{t+1}^2).
\end{equation}
\end{lemma}

\begin{proof}
Let $P_t^2 \in \mathcal{P}^2_t$ be a given set-valued realization of the information state $\Pi_t^2$ at time $t$. Then, for the realizations $\gamma_t^{[2,1]}$, $\hat{u}_t^{2:3}$ and $z_{t+1}^2$ of $\Gamma_t^{[2,1]}$, $\hat{U}_t^{2:3}$ and $Z_{t+1}^2$, subsystem $2$ can eliminate some possible elements in $P_t^2$ at the start of time $t+1$. Specifically, we define an interim set  
\begin{multline}
    Q_t^2 := \big\{(\hat{x}_t, P_t^1) \in P_t^2 ~|~ z_{t+1}^2 = \hat{h}_t^2(\hat{x}_t^2, \hat{u}_t^{1:3}), \\
     \hat{u}_t^1 = \gamma_t^{[2,1]}(\pi_t^1)\big\},
\end{multline}
which is completely determined given the set $P_t^2$ and the realizations $\{\gamma_t^{[2,1]}, \hat{u}_t^{2:3}, z_{t+1}^2\}$.
Next, we derive the realization $P_{t+1}^2$ of the information state $\Pi_{t+1}^2$ using set $Q_t^2$ as
\begin{multline}
    P_{t+1}^2= \big\{\hat{x}_{t+1} \in \hat{\mathcal{X}}_{t+1}, P_{t+1}^1 \in \mathcal{P}_{t+1}^1 ~|~ \hat{x}_{t+1} = \hat{f}_t\big(\hat{x}_{t}, \\
    \hat{u}_t^{1:3}, w_t, v_{t+1}^{1:3}\big), \; P_{t+1}^1 = \tilde{f}_t^1(P_t^1, \hat{u}_t^{1:3}, z_{t+1}^1), z_{t+1}^1 = \hat{h}_t^1(\hat{x}_t, \\
    \hat{u}_t^{1:3}), \; \hat{u}_t^1 = \gamma_t^{[2,1]}(P_t^1) \text{ for all } [(\hat{x}_t, P_t^1) \in Q_t^2, \\
     w_t \in \mathcal{W}_t, v_{t+1}^{n} \in \mathcal{V}_{t+1}^{n}, n = 1,2,3 ] \big\}. \label{tildef_t_2_def}
\end{multline}
Thus, we can define an appropriate function $\tilde{f}_t^2(\cdot)$ using \eqref{tildef_t_2_def} such that that $P_{t+1}^2  = \tilde{f}_t^2(P_t^2, \gamma_t^{[2,1]}, \hat{u}_t^{2:3}, z_{t+1}^2)$.
\end{proof}

Lemma \ref{pi_2_evol} establishes that the evolution of $\Pi_t^2$ is Markovian. Thus, we can state the following result for Problem \ref{problem_4} using the standard centralized DP in \cite{bertsekas1973sufficiently}.

\begin{theorem} \label{pbp_struct_result_2}
    For a given partial strategy $\hat{\boldsymbol{g}}^{3}$ of subsystem $3$ in Problem \ref{problem_4}, without loss of optimality, we can restrict attention to prescription strategies $\boldsymbol{\psi}^{*2}$ and partial strategies $\hat{\boldsymbol{g}}^{*2}$ with the structural form
    \begin{gather} \label{eq_pbp_struct_result_2}
        \Gamma_t^{[2,1]} = \psi_t^{*[2,1]}\big(\Pi_t^2, C_t^3\big), \\
        \hat{U}_t^{2} = \hat{g}_t^{*2}\big(\Pi_t^2, C_t^3\big), \quad t = 0,\dots,T. \label{eq_thm_2_2}
    \end{gather}
\end{theorem}

\begin{proof}
This result follows from standard arguments for centralized minimax control problems \cite{bertsekas1973sufficiently}.
\end{proof}

Using an optimal prescription strategy which satisfies \eqref{eq_pbp_struct_result_2}, we construct an optimal partial strategy $\hat{\boldsymbol{g}}^{*1}$ for Problem \ref{problem_2} such as $\hat{g}_t^{*1}(\cdot, \Pi_t^2, C_t^3) := \psi_t^{*[2,1]}(\Pi_t^2, C_t^3)(\cdot)$ and recall from Remark \ref{remark_4} that they yield the same performance. Thus, we can derive an optimal partial strategy $\hat{\boldsymbol{g}}^{*1}$ in the form
\begin{align}
    \hat{U}_t^1 &= \hat{g}_t^{*1}(\Pi_t^{1:2},C_t^3), \quad t =0,\dots,T. \label{eq_thm_2_1}
\end{align}

\subsection{Analysis for Subsystem $3$}

In this subsection, we restrict attention to partial strategies $\hat{\boldsymbol{g}}^1$ and $\hat{\boldsymbol{g}}^2$ of the form in \eqref{eq_thm_2_1} and \eqref{eq_thm_2_2}, respectively, and derive a structural form for the partial strategy of subsystem $3$. Given $\hat{\boldsymbol{g}}^1$ and $\hat{\boldsymbol{g}}^2$, subsystem $3$ cannot generate the control actions $\hat{U}_t^{1:2}$ for any $t = 0,\dots,T$, because it cannot access the information states $\Pi_t^1$ and $\Pi_t^2$.
Instead, subsystem $3$ generates a prescription for both subsystems $1$ and $2$ in a manner similar to Section III-B. Specifically, for each $t$, a prescription of subsystem $3$ for subsystem $1$ is a mapping $\Gamma_t^{[3,1]}: \mathcal{P}_t^1 \times \mathcal{P}_t^2 \to \hat{\mathcal{U}}_t^{1}$ which takes values in a finite set $\mathcal{F}_t^{[3,1]}$ and that for subsystem $2$ is a mapping $\Gamma_t^{[3,2]}: \mathcal{P}_t^2 \to \hat{\mathcal{U}}_t^{2}$ which takes values in a finite set $\mathcal{F}_t^{[3,2]}$. At each $t=0,\dots,T$, subsystem $3$ generates a prescription for each $n =1,2$ using a prescription law $\psi_t^{[3,n]}: \mathcal{C}_t^3 \to \mathcal{F}_t^{[3,n]}$, which yields $\Gamma_t^{[3,n]} = \psi_t^{[3,n]}(C_t^3)$. We define the prescription strategy of subsystem $3$ as $\boldsymbol{\psi}^3 := (\psi_t^{[3,n]} : n = 1,2, \; t=0,\dots,T)$. In this formulation, after receiving the prescriptions of subsystem $3$, subsystem $1$ generates its partial action as $\hat{U}_t^1 = \Gamma_t^{[3,1]}(\Pi_t^{1:2})$ and subsystem $2$ generates its partial action as $\hat{U}_t^2 = \Gamma_t^{[3,2]}(\Pi_t^2)$. Thus, subsystem $3$ is the sole decision maker with a complete action $\Theta_t^3 := (\Gamma_t^{[3,1]},\Gamma_t^{[3,2]},\hat{U}_t^3)$. Next, we construct a new state for subsystem $3$ as $S_t^3 := \{\hat{X}_t, \Pi_t^1, \Pi_t^2\}$ for all $t$, which takes values in a finite collection of sets $\mathcal{S}_t^3$. Using the same arguments as Lemma \ref{lem_pbp_suff}, we can construct a state evolution function $\bar{f}^3_t(\cdot)$ such that ${S}^3_{t+1} = \bar{f}^3_{t}({S}^3_t, \Theta_t^3, W_{t}, V_{t+1}^{1:3})$, and an observation rule $\bar{h}^{{3}}_t(\cdot)$ which yields $Z^3_{t+1} = \bar{h}^3_t({S}^3_t, \Theta_t^3)$ to obtain a centralized problem with state $S_t^3$, observation $Z_t^3$ and complete action $\Theta_t^3$ for all $t$. Furthermore, we can construct a cost function $\bar{d}_T^3(\cdot)$ which yields a terminal cost $\bar{d}_T^3(S_T^3) := \hat{d}_T(\hat{X}_T)$ and a performance criterion $\mathcal{J}^3(\boldsymbol{\psi}^3, \hat{\boldsymbol{g}}^3) := \max_{X_0, W_{0:T}, V_{0:T}^{1:3}} \bar{d}_T^3(S_T^3)$.

\begin{problem} \label{problem_5}
The problem for subsystem $3$ is $\inf_{\boldsymbol{\psi}^3, \hat{\boldsymbol{g}}^3} \mathcal{J}^3(\boldsymbol{\psi}^3, \hat{\boldsymbol{g}}^3)$, given the feasible sets $\big\{\mathcal{X}_0,\mathcal{W}_{t},\mathcal{V}_{t}^{n} : n = 1,2,3, \; t = 0,\dots,T\big\}$, the cost function $\bar{d}^3_T$ and the dynamics $\big\{\bar{f}^3_t,\bar{h}_t^3: t=0,\dots,T\big\}$.
\end{problem}

Using the same arguments as in Lemma \ref{lem_hatg_g_relation} and Remark \ref{remark_4}, we conclude that we can generate partial actions $\hat{U}_t^n$, for $n = 1,2$ and $t =0,\dots,T$, equivalently using either a prescription
strategy $\boldsymbol{\psi}^3$ or an appropriate partial control strategy $\hat{\boldsymbol{g}}^n$. In Problem \ref{problem_5}, at each $t$, subsystem $3$ must estimate the entire state $S_t^3$. Thus, for the realizations $c_t^3$ and $\theta_{0:t-1}^{3}$ of $C_t^3$ and $\Theta_{0:t-1}^{3}$, respectively, the realized information state for subsystem $3$ is
\begin{multline} \label{pi_3_def}
    P_t^3 
    := \big\{\hat{x}_t \in \hat{\mathcal{X}}_t, 
    P_t^1 \in \mathcal{P}_t^1, P_t^2 \in \mathcal{P}_t^2 ~
    |~ \exists \; \big(x_0 \in \mathcal{X}_0, w_{0:t-1} \\ \in \prod_{\ell=0}^{t-1} \mathcal{W}_{\ell},
    v_{0:t}^{n} \in \prod_{\ell=0}^{t} \mathcal{V}_{\ell}^{n}, 
    n = 1,2,3 \big) \text{ s.t. } s_{\ell+1}^3 = \bar{f}_\ell^3 \big(s_{\ell}^3, \theta_\ell^3,\\  w_\ell, v_{\ell+1}^{1:3}\big), z_{\ell+1}^3 = \bar{h}_\ell^3\big(
    s_{\ell}^3, \theta_\ell^3,\big), \ell = 0, \dots, t-1
    \big\}.
\end{multline}
We define the information state of subsystem $3$ as a set-valued uncertain variable which takes values in a finite collection of feasible sets $\mathcal{P}^3_t \subseteq 2^{\hat{\mathcal{X}}_t} \times 2^{\mathcal{P}^1_t} \times 2^{\mathcal{P}^2_t}$. Next, we show that $\Pi_t^2$ is independent of $\hat{\boldsymbol{g}}$.

\begin{lemma} \label{pi_3_evol}
At each $t = 0,\dots,T-1$, there exists a function $\tilde{f}_t^3(\cdot)$ independent from $\hat{\boldsymbol{g}}$, such that
\begin{equation} \label{pi_3_relation}
    \Pi_{t+1}^3 = \tilde{f}_{t}^3(\Pi_t^3,\Gamma_t^{[3,1]},\Gamma_t^{[3,2]},\hat{U}_t^{3},Z_{t+1}^3).
\end{equation}
\end{lemma}

\begin{proof}
The proof is the same as the proof of Lemma \ref{pi_2_evol}. Due to space constraints, it is omitted.
\end{proof}

Lemma \ref{pi_3_relation} establishes that the evolution of the $\Pi_t^3$ is Markovian. Thus, we can use the standard centralized DP decomposition from \cite{bertsekas1973sufficiently} to state the following result.

\begin{theorem} \label{pbp_struct_result_3}
    In Problem \ref{problem_5}, without loss of optimality, we can restrict attention to prescription strategies $\boldsymbol{\psi}^{*3}$ and partial strategies $\hat{\boldsymbol{g}}^{*3}$ with the structural form
    \begin{align} \label{eq_pbp_struct_result_3}
        \Gamma_t^{[3,n]} &= \psi_t^{*[3,n]}\big(\Pi_t^3\big), \quad n =1,2,\\
        \hat{U}_t^{3} &= \hat{g}_t^{*3}\big(\Pi_t^3\big), \quad t = 0,\dots,T.
    \end{align}
\end{theorem}

\begin{proof}
This proof follows similar arguments for centralized minimax control problems \cite{bertsekas1973sufficiently}. Due to space constraints, it is omitted.
\end{proof}

As a result of Theorem \ref{pbp_struct_result_3}, we conclude that we can derive an optimal partial strategy $\hat{\boldsymbol{g}}^*$ for Problem \ref{problem_2} with the form
\begin{align}
    \hat{U}_t^n &= \hat{g}_t^{*n}(\Pi_t^{n:3}), \quad n = 1,2,3, \quad t =0,\dots,T. \label{eq_thm_3_1}
\end{align}

\subsection{Result for $N$ Subsystems}

The results equivalent to Theorems \ref{pbp_struct_result} -  \ref{pbp_struct_result_3} for $N \in \mathbb{N}$ subsystems can be proven using arguments similar to those for $3$ subsystems and mathematical induction. The steps for the mathematical induction are detailed in Appendix A of our online preprint \cite{Dave2021arxiv}. Here, due to space constraints, we simply state the main result for Problem \ref{problem_2} and Problem \ref{problem_1}.

\begin{theorem} \label{struct_result_k}
In Problem \ref{problem_2}, without loss of optimality, we can restrict attention to partial strategies $\hat{\boldsymbol{g}}^*$ with the form
\begin{gather}
    \hat{U}_t^n = \hat{g}_t^{*n}(\Pi_t^{n:N}), \quad t = 0,\dots,T. \label{eq_struct_result_k}
\end{gather}
Furthermore, in Problem \ref{problem_1}, there exists an optimal control strategy $\boldsymbol{g}^{*k,n}$ for each $k \in \mathcal{K}^n$ and $n \in \mathcal{N}$, with the structural form
\begin{gather}
    U_t^{k,n} = g_t^{*k,n}(Y_t^{k,n},L_t^{k,n}, \Pi_t^{n:N}), \quad t = 0,\dots,T. \label{eq_struct_result_k_agent}
\end{gather}
\end{theorem}

\begin{proof}
The proof follows from the arguments detailed in Apprendix A of \cite{Dave2021arxiv}.
\end{proof}

\begin{remark} \label{remark_common_info}
The structural form of optimal control strategies in \eqref{eq_struct_result_k_agent} cannot be obtained by a direct application of the common information approach \cite{gagrani2017decentralized}. Note that the domains of our optimal strategies do not grow in size with time when the feasible sets $\mathcal{X}_t$ and $\mathcal{L}_t^{k,n}$ do not grow in size with time for all $k \in \mathcal{K}^n$ and $n \in \mathcal{N}$. This improves the computational tractability of our approach for larger values of horizon $T$. In contrast, the common information approach considers $C_t^1$ as a part of the private information of all agents in $\mathcal{K}^1$ and $C_t^1$ grows in size with time. Thus, the domains of optimal strategies increase in size with time for $N \geq 2$.  
\end{remark}

\subsection{Dynamic Programming Decomposition} \label{subsection:DP}

In this subsection, we construct a DP decomposition for the optimal prescription and partial strategy of subsystem $N$. Let $\theta_t^N = (\gamma_t^{[N,1]}, \dots, \gamma_t^{[N,N-1]}, \hat{u}_t^N)$ be the realization of $\Theta_t^N$ for all $t=0,\dots,T$. Then, for each possible information state $P_T^N \in \mathcal{P}_T^N$, we define a value function at time $T$ as
$V_T(P_T^N) := \max_{s_T^N \in P_T^N} \bar{d}_T^N(s_T^N)$. Furthermore, for each possible information state $P_t^N \in \mathcal{P}_t^N$, for all $t =0,$ $\dots,T-1$, we iteratively define the value functions
\begin{align}
    V_t(P_t^N) := \min_{\substack{\theta_t^N}}  
    \max_{\substack{z_{t+1}^N \in \tilde{\mathcal{Z}}_{t+1}^N(P_t^N, \theta_t^N)}} V_{t+1}[\tilde{f}_t^N(P_t^N, \theta_t^N, z_{t+1}^N)], \label{value_t}
\end{align}
where the set of feasible values of $z_{t+1}^N$ is given by
\begin{multline*}
    \tilde{\mathcal{Z}}_{t+1}^N(P_t^N, \theta_t^N) := \{z_{t+1}^N \in \mathcal{Z}_{t+1}^N ~|~ z_{t+1}^N = \bar{h}_t^N(s_t,\theta_t^N, w_t, \\ v_{t+1}^{1:N}) \text{ for all } s_t^N \in P_t^N, w_t \in \mathcal{W}_t, v_{t+1}^{n} \in \mathcal{V}_{t+1}^{n}, n \in \mathcal{N}\}.
\end{multline*}
The prescription law at time $t$ for each $n < N$ is $\gamma_t^{*[N,n]} = \psi_t^{*[N,n]}(P_t^N)$ and the partial control law for subsystem $N$ is $\hat{u}_t^{*N} = \hat{g}_t^{*N}(P_t^N)$, i.e., they are the $\arg\inf$ in the RHS of \eqref{value_t}. The corresponding prescription strategy $\boldsymbol{\psi}^{*N}$ and partial strategy $\hat{\boldsymbol{g}}^{*N}$ can be shown to be optimal for $\mathcal{J}^N(\boldsymbol{\psi}^N, \hat{\boldsymbol{g}}^N)$ using standard arguments \cite{bertsekas1973sufficiently, gagrani2017decentralized}.


\begin{remark}
After solving the DP, we can construct optimal control strategies for Problem \ref{problem_1} 
in a manner similar to Lemma \ref{lem_hatg_g_relation}.
\end{remark}

\begin{remark}
When applying our results to an arbitrary decentralized system, there may be multiple feasible ways to allocate the agents to subsystems. For example, consider a three agent system with one-directional communication from agent $3$ to agent $2$ and from agent $2$ to agent $1$. We can either allocate all agents to one subsystem or allocate each agent to a unique subsystem. Both of these are valid options because they both ensure that the common information of each subsystem is nested within the common information of all preceding subsystems. 
However, Theorem \ref{struct_result_k} leads to a different DP for different allocations.
Thus, a system designer must decide on an allocation before deriving the optimal control strategies. In general, our DP performs better when more agents have private information which does not grow in size with time. Thus, in the three agent system with one-directional communication, 
allocating each agent to a unique subsystem is better than allocating all agents to one subsystem.
\end{remark}

\subsection{Extension to Additive Cost}

Consider a variation of our problem where the system incurs a cost $d_t(X_t, U_t^{1:N})$ at each $t=0,\dots,T$ and the performance is measured by the worst-case total cost
\begin{gather} \label{per_cri_tot}
    \Xi (\boldsymbol{g}) := \max_{X_0, W_{0:T}, V_{0:T}^{1:N}} \sum_{t=0}^{T-1} d_t(X_t, U_t^{1:N}) + d_T(X_T).
\end{gather}
We can transform the additive cost in \eqref{per_cri_tot} into a terminal cost using a technique from \cite{bertsekas1973sufficiently, gagrani2017decentralized}. At each $t=0,\dots,T$, we define an uncertain variable $A_t := \sum_{\ell=0}^{t-1} d_\ell(X_\ell, U_\ell^{1:N}) \in \mathcal{A}_t$ which tracks the cost incurred by the system up to time $t$. Note that $A_0 = 0$. 
Then, at each $t$, we consider an augmented state for the system, $\{X_t, A_t\}$ and note that the augmented state at time $t+1$, $\{X_{t+1}, A_{t+1}\}$, evolves as a function of $\{X_{t}, A_{t}\}$, the control actions $U_t^{1:N}$ and the disturbance $W_t$. Then the total cost is simply given by the terminal cost $A_T + d_T(X_T)$ and thus, we can apply our results to this equivalent terminal cost problem. This extension to total cost problems does present a challenge because the information state at time $t$ takes values in a subset of the power set $2^{\mathcal{X}_t \times \mathcal{A}_t}$, which may grow in size with time as $\mathcal{A}_t$ generally grows in size with time. We plan to address this in the future by exploring alternative augmented states.

\section{Numerical Example} \label{section:example}

In this subsection, we validate our results with a simple example. We consider two agents who seek to surround a target. The agents and target can each move along a linear grid with $\Lambda \in \mathbb{N}$ points. Starting at $X^0_0$, the position of the target is updated at each $t=0,\dots,T$ as 
\begin{gather}
    X_{t+1}^0 =
    \begin{aligned}
    \begin{cases}
        X_t^0 + W^0_t, \quad &\text{if } 1 \leq X_t^0 + W_t^0 \leq \Lambda, \\
        \Lambda, \quad &\text{if } X_t^0 + W_t^0 \geq \Lambda, \\
        1, \quad &\text{if } X_t^0 + W_t^0 \leq 1,
    \end{cases}
    \end{aligned}
\end{gather}
with a disturbance $W_t^0 \in \{-1, 0, 1\}$. Each agent occupies a separate subsystem, and thus, we refer to an agent simply by its subsystem $n = 1,2$. At each $t$, each agent $n$ selects an action $U_t^n \in \{-1,0,1\}$, and updates its position as 
\begin{gather}
    X_{t+1}^n =
    \begin{aligned}
    \begin{cases}
        X_t^n + U_t^n, \quad &\text{if } 1 \leq X_t^n + U_t^n \leq \Lambda, \\
        \Lambda, \quad &\text{if } X_t^n + U_t^n \geq \Lambda, \\
        1, \quad &\text{if } X_t^n + U_t^n \leq 1.
    \end{cases}
    \end{aligned}
\end{gather}
The positions of both agents are observed perfectly by both the agents. Additionally, agent $2$ perfectly observes the position of the target, but can only communicate this to agent $1$ with a delay of $1$ time step. Agent $1$ also receives a noisy observation of the target's position $Y_t^1  = \max \{1, \min \{ X_t^0 + V^1_t ,\Lambda\}\}$, with a noise $V^1_t \in \{-1,0\}$. Agent $1$ has faulty equipment and cannot transmit their observations to agent $2$. Thus, for all $t=0,\dots,T$, the information structure is given by $C_t^1 = \{X_{0:t-1}^0, Y_{0:t-1}^1, U_{0:t-1}^{1:2}, X_{0:t}^{1:2}\}$ and $L_t^1 = \{Y_t^1\}$ for agent $1$, and $C_t^2 = \{X_{0:t-1}^0, U_{0:t-1}^2, X_{0:t}^{1:2}\}$ and $L_t^2 = \{X_t^2\}$ for agent $2$. At the onset, both agents receive a common observation $Y_0$ implying that the target's starting position is in the set $\{Y_0-1, Y_0, Y_0+1\}$. The terminal cost is 
\begin{gather}
    d_T(X_T^{0:3})
    =
    \begin{aligned}
    \begin{cases}
        \sum_{n=1}^2|X_T^0 - X_T^n|, \quad &\text{ if } I^s = 1, \\
        D + \sum_{n=1}^2|X_T^0 - X_T^n|, \quad &\text { if } I^s = 0,
    \end{cases}
    \end{aligned}
\end{gather}
where $I^{s} \in \{0,1\}$ indicates if the agents have successfully surrounded the target, $\sum_{n=1}^2|X_T^0 - X_T^n|$ penalizes the distance of the agents from the target, and $D > 0$ is a penalty for failing to surround the target. 
We summarize in Table \ref{summary_table} the optimal worst-case performance achieved for different realizations of the initial conditions $\{x_0^1, x_0^2, y_0\}$, with and without using information states in the DP. We note that their values are the same, which validates our results. 

\begin{table}[ht]
\centering
\caption{Optimal cost for $\Lambda = 8$, $T=3$, and $D = 10$.}
\label{summary_table}
\begin{tabular}{ p{60pt} p{70pt} p{70} }
 \hline 
 \arrayrulecolor{white}
 \hline
 \hline
 \arrayrulecolor{black}
 \multicolumn{1}{c}{$\{x_0^1, x_0^2, y_0\}$} &  \multicolumn{1}{c}{With Information States} & \multicolumn{1}{c}{Without Information States}\\[2pt]
 \hline
 \arrayrulecolor{white}
 \hline
 \hline
 \arrayrulecolor{black}
 \multicolumn{1}{c}{$\{8,8,2\}$} & \multicolumn{1}{c}{$18$} & \multicolumn{1}{c}{$18$}\\ [2pt]
 \multicolumn{1}{c}{$\{3,6,7\}$} & \multicolumn{1}{c}{$4$} & \multicolumn{1}{c}{$4$}\\[2pt]
 \multicolumn{1}{c}{$\{3,3,4\}$} & \multicolumn{1}{c}{$14$} & \multicolumn{1}{c}{$14$}\\[2pt]
 \multicolumn{1}{c}{$\{3,5,8\}$} & \multicolumn{1}{c}{$4$} & \multicolumn{1}{c}{$4$}\\[2pt]
 \hline
\end{tabular}
\end{table}

\section{Conclusion} \label{section:conclusion}

In this paper, we introduced a general model for decentralized control with nested subsystems and a worst-case shared cost. The disturbances to the system were modeled using uncertain variables. We derived a structural form for optimal control strategies in systems with a terminal cost criterion and observed that it can be used to improve computational tractability for long time horizons. Finally, we presented a DP to derive the optimal strategies and validated our results with an example. A potential direction for future research includes specializing our results for systems with specific dynamics. Another important direction of future research is identifying approximation techniques which can improve computational tractability with a bounded loss in optimal performance. Finally, we believe that our results can also be utilized in applications of decentralized reinforcement learning problems with a worst-case cost.

\bibliographystyle{ieeetr}
\bibliography{References}

\section*{Appendix A - $N$ Subsystems}

In this subsection, we describe the steps to prove our results using mathematical induction for $N \in \mathbb{N}$ subsystems. The analysis for subsystem $1$ is the same in the presence of $N$ subsystems as in Section \ref{pbp_agent_1}. As in \eqref{pi_1_def}, we can define the information state for subsystem $1$ at any $t$ as the set-valued uncertain variable $\Pi_t^1$ which contains feasible values of $\hat{X}_t \in \hat{\mathcal{X}}^1_t$ given $\{C_t^1, \hat{U}_{0:t-1}^{1:N}\}$, and takes values in $\mathcal{P}_t^1 \subseteq 2^{\hat{\mathcal{X}_t}}$.
Using the same arguments as Theorem \ref{pbp_struct_result}, 
in Problem \ref{problem_2}, without loss of optimality, we can restrict attention to partial strategies $\hat{\boldsymbol{g}}^{*1}$ with the form
\begin{gather} \label{agent_1_in_K}
    \hat{U}_t^1 = \hat{g}_t^{*1}(\Pi_t^1, C_t^2), \quad t = 0,\dots,T.
\end{gather}
This forms the basis of our mathematical induction. Given \eqref{agent_1_in_K}, we can iteratively prove the results for each subsystem $n \in \{2,\dots,N\}$. For any $n \in \mathcal{N}$, we consider the induction hypotheses for all $m \in \mathcal{N}$ where $m < n$: 

\textit{Hypothesis 1)} The information state is well defined for all $t$ as an appropriate set-valued uncertain variable $\Pi_t^m$ which takes values in a space $\mathcal{P}_t^m \subseteq 2^{\hat{\mathcal{X}}_t} \times 2^{\mathcal{P}_t^{1}} \times \dots \times 2^{\mathcal{P}_t^{m-1}}$.

\textit{Hypothesis 2)} Without loss of optimality in Problem \ref{problem_2}, we can restrict attention to partial strategies $\hat{\boldsymbol{g}}^{*}$ with the form $\hat{U}_t^{m} = \hat{g}_t^{*m}(\Pi_t^{m : n-1}, C_t^n)$ for all $m < n$ and $t=0,\dots,T$.

\begin{remark}
In hypothesis 2, we consider that a structural form using information states has already been derived for each subsystem $m < n$. Note that the form in \eqref{agent_1_in_K} for $\hat{\boldsymbol{g}}^{*1}$ is consistent with this for $n = 2$. We later show that the structural form of strategies derived for each subsystem $n \in \mathcal{N}$ is also consistent with hypothesis 2.
\end{remark}

Given the two hypotheses, the proof by mathematical induction follows from the following steps:

\textit{Step 1)} For all $t=0,\dots,T$, we define a prescription by subsystem $n$ for all $m < n$ as a mapping $\Gamma_t^{[n,m]} : \mathcal{P}_t^m\times$ $ \dots \times \mathcal{P}_t^{n-1} \to \hat{\mathcal{U}}_t^m$
which takes values in a finite set $\mathcal{F}_t^{[n,m]}$. Each prescription $\Gamma_t^{[n,m]}$ is generated using a prescription law $\psi_t^{[n,m]}: \mathcal{C}_t^n \to \mathcal{F}_t^{[n,m]}$ which yields $\Gamma_t^{[n,m]} = \psi_t^{[n,m]}(C_t^n)$. We define the prescription strategy of subsystem $n$ as $\boldsymbol{\psi}^n := (\psi_t^{[n,m]}: m < n, t = 0,\dots, T)$. Furthermore, we define the complete action of subsystem $n$ as $\Theta_t^n := (\Gamma_t^{[n,1]}, \dots, \Gamma_t^{[n,n-1]},\hat{U}_t^n)$ for all $t=0,\dots,T$. 

\textit{Step 2)} 
We fix the partial control strategies $\hat{\boldsymbol{g}}^{n+1}, \dots, \hat{\boldsymbol{g}}^{N}$ of all subsystems $n+1,\dots,N$, and use the person-by-person approach to derive a structural form for $\hat{\boldsymbol{g}}^{*n}$. We construct a state for subsystem $n$ as $S_t^n := \{\hat{X}_t,\Pi_t^1, \dots, \Pi_t^{n-1}, C_t^{n+1}\},$
which takes values in a finite collection of sets $\mathcal{S}_t^n$. Note that $C_t^{N+1} := \emptyset$. As in Lemma \ref{lem_pbp_suff}, we can also construct a state evolution function $\bar{f}^n_t(\cdot)$ such that ${S}^n_{t+1} = \bar{f}^n_{t}({S}^n_t, \Theta_t^{n}, W_{t}, V_{t+1}^{1:N})$, and an observation rule $\bar{h}^n_t(\cdot)$ which yields $Z^n_{t+1} = \bar{h}^n_t({S}^n_t, \Theta_t^{n})$ to obtain a centralized problem with state $S_t^n$, observation $Z_t^n$ and complete action $\Theta_t^n$ for all $t$. Furthermore, we can construct a cost function $\bar{d}_T^n(\cdot)$ which yields a terminal cost $\bar{d}_T^n(S_T^n) := \hat{d}_T(\hat{X}_T)$ and a performance criterion
$\mathcal{J}^n(\boldsymbol{\psi}^n, \hat{\boldsymbol{g}}^n) = \max_{X_0, W_{0:T}, V_{0:T}^{1:N}} \bar{d}_T^n(S_T^n)$.

\textit{Step 3)} Using the same arguments as Lemma \ref{lem_hatg_g_relation}, we can prove that for each prescription strategy $\boldsymbol{\psi}^n$, we can construct partial strategies $(\hat{\boldsymbol{g}}^1,\dots, \hat{\boldsymbol{g}}^{n-1})$ that lead to the same partial actions $(\hat{U}_t^{1},\dots,\hat{U}_t^{n-1})$ for all $t$, and vice versa. Furthermore, this construction ensures that after fixing strategies $\hat{\boldsymbol{g}}^{n+1}, \dots, \hat{\boldsymbol{g}}^{N}$, it holds that $\mathcal{J}^n(\boldsymbol{\psi}^n, \hat{\boldsymbol{g}}^n) = \mathcal{J}(\hat{\boldsymbol{g}})$.


\textit{Step 4)} 
In the constructed centralized problem of $\min_{\boldsymbol{\psi}^n, \hat{\boldsymbol{g}}^n}\mathcal{J}^n(\boldsymbol{\psi}^n, \hat{\boldsymbol{g}}^n)$, the unobserved components in $S_t^n$ need to be estimated at each $t=0,\dots,T$ using an information state. For realizations $c_t^n$ and $\theta_{0:t}^n$ of $C_t^n$ and $\Theta_{0:t}^n$, respectively, the realization of the information state is
$P_t^n := \big\{\hat{x}_t \in \hat{\mathcal{X}}_t, 
    P_t^1 \in \mathcal{P}_t^1, \dots, P_t^{n-1} \in \mathcal{P}_t^{n-1} ~
    |~ \exists \big(x_0 \in \mathcal{X}_0, w_{0:t-1} \in \prod_{\ell=0}^{t-1} \mathcal{W}_{\ell},
    v_{0:t}^{n} \in \prod_{\ell=0}^{t} \mathcal{V}_{\ell}^{n}, \; n \in \mathcal{N} \big)$ s.t. $s_{\ell+1}^n = \bar{f}_\ell^n\big(s_{\ell}^n,\theta_\ell^{n}, w_\ell, v_{\ell+1}^{1:N}\big),
    z_{\ell+1}^n = \bar{h}_\ell^n\big(s_{\ell}^n,\theta_\ell^{n}\big), \; \ell = 0, \dots, t-1
    \big\}$.
Thus, the information state for subsystem $n$ at time $t$ is a set-valued uncertain variable $\Pi_t^n$ which takes values in a finite set $\mathcal{P}_t^n \subseteq 2^{\hat{\mathcal{X}}_t} \times 2^{\mathcal{P}_t^{1}} \times \dots \times 2^{\mathcal{P}_t^{n-1}}$. Using the same sequence of arguments as Lemma \ref{pi_2_evol} we can show that at each time $t$, there exists a function $\tilde{f}_t^{{n}}(\cdot)$ independent of the partial control strategy $\hat{\boldsymbol{g}}$, such that
$\Pi_{t+1}^{{n}} = \Tilde{f}_{t}^{{n}}(\Pi_t^{{n}},\Theta_t^{{n}}, \hat{U}_t^{n+1:N}, Z_{t+1}^{{n}}).$
Furthermore, using the same arguments as Theorems \ref{pbp_struct_result_2} and \ref{pbp_struct_result_3}, we conclude that without loss of optimality, we can restrict attention to prescription strategies $\boldsymbol{\psi}^{*n}$ and partial strategies $\hat{\boldsymbol{g}}^{*n}$ with the structural form
\begin{align} \label{eq_struct_k}
    \Gamma_t^{[n,m]} &= \psi_t^{*[n,m]}(\Pi_t^n, C_t^{n+1}), \; m < n, \; t =0,\dots,T, \\
    \hat{U}_t^n &= g_t^{*n}(\Pi_t^n, C_t^{n+1}), \quad t =0,\dots, T. \label{eq_struct_k_2}
\end{align}
The results \eqref{eq_struct_k} and \eqref{eq_struct_k_2} are consistent with the induction hypotheses and complete the proof by mathematical induction. Then, the first result in Theorem \ref{struct_result_k} follows from \eqref{eq_struct_k} and \eqref{eq_struct_k_2} by constructing a partial control law $\hat{g}_t^{*n}(\Pi_t^{n:N}) := \psi_t^{*[N,n]}(\Pi_t^N)(\Pi_t^{n:N-1})$, for all $n \in \mathcal{N}$ and $t=0,\dots,T$. The second result follows by constructing an appropriate control strategy $\boldsymbol{g}^{*}$ from $\hat{\boldsymbol{g}}^{*}$ using Lemma \ref{lem_hatg_g_relation}.


\end{document}